\documentclass[12pt,leqno]{amsart}
\usepackage{amsmath,amstext,amsthm,amssymb,amsxtra}
\usepackage{graphicx}
\usepackage{epsfig}
\usepackage{amsfonts}
\usepackage{latexsym}

\usepackage[left=1.8cm,top=2.3cm,right=1.8cm]{geometry}

\newtheorem{thm}{Theorem}[section]
\newtheorem{prop}[thm]{Proposition}
\newtheorem{cor}[thm]{Corollary}
\newtheorem{lem}[thm]{Lemma}

\theoremstyle{definition}

\newtheorem{rem}[thm]{Remark}
\newtheorem{defn}[thm]{Definition}

\def\vint{\mathop{\mathchoice%
          {\setbox0\hbox{$\displaystyle\intop$}\kern 0.22\wd0%
           \vcenter{\hrule width 0.6\wd0}\kern -0.82\wd0}%
          {\setbox0\hbox{$\textstyle\intop$}\kern 0.2\wd0%
           \vcenter{\hrule width 0.6\wd0}\kern -0.8\wd0}%
          {\setbox0\hbox{$\scriptstyle\intop$}\kern 0.2\wd0%
           \vcenter{\hrule width 0.6\wd0}\kern -0.8\wd0}%
          {\setbox0\hbox{$\scriptscriptstyle\intop$}\kern 0.2\wd0%
           \vcenter{\hrule width 0.6\wd0}\kern -0.8\wd0}}%
          \mathopen{}\int}

\DeclareMathOperator{\diam}{diam}
\DeclareMathOperator*{\esssup}{ess\, sup}

\DeclareMathOperator{\dima}{dim_A}
\DeclareMathOperator{\udima}{\overline{dim}_A}

\DeclareMathOperator{\dimh}{dim_H}

\DeclareMathOperator{\lcodima}{\underline{co\,dim}_A}
\DeclareMathOperator{\codimh}{co\,dim_H}

\newcommand{\dist}{\mathit d}

      \newcommand{\N}{{\mathbb N}}
\newcommand{\Z}{{\mathbb Z}}      \newcommand{\R}{{\mathbb R}}

\def\dist{\qopname\relax o{dist}}

\newcommand{\Ha}{{\mathcal H}}

\newcommand{\sub}{\subset}
\newcommand{\rad}{\operatorname{rad}}

\newcommand{\Lip}{\operatorname{Lip}}

\begin{document}

\title[$A_p$-properties of distance functions]
{Muckenhoupt $A_p$-properties of distance functions and\\
applications to Hardy--Sobolev -type inequalities}

\author[B. Dyda]{Bart{\l}omiej Dyda}   %
\address[B.D.]{Faculty of Pure and Applied Mathematics, Wroc{\l}aw University of Science and Technology, ul. Wybrze\.{z}e Wyspia\'{n}skiego 27, 50-370 Wroc{\l}aw, Poland}
\email{bartlomiej.dyda@pwr.edu.pl}

\author[L. Ihnatsyeva]{Lizaveta Ihnatsyeva}   
\address[L.I.]{Department of Mathematics, Kansas State University, Manhattan, KS 66506, USA}
\email{ihnatsyeva@math.ksu.edu}

\author[J. Lehrb\"ack]{Juha Lehrb\"ack}   
\address[J.L.]{Department of Mathematics and Statistics, P.O. Box 35, FI-40014 University of Jyvaskyla, Finland}
\email{juha.lehrback@jyu.fi}

\author[H. Tuominen]{Heli Tuominen}   
\address[H.T.]{Department of Mathematics and Statistics, P.O. Box 35, FI-40014 University of Jyvaskyla, Finland} 
\email{heli.m.tuominen@jyu.fi}

\author[A. V. V\"ah\"akangas]{Antti V. V\"ah\"akangas}
\address[A.V.V.]{Department of Mathematics and Statistics, P.O. Box 35, FI-40014 University of Jyvaskyla, Finland}
 \email{antti.vahakangas@iki.fi}


\begin{abstract}
Let $X$ be a metric space equipped with a doubling measure. We consider weights $w(x)=\dist(x,E)^{-\alpha}$, where $E$ is a closed set in $X$ and $\alpha\in\R$. We establish sharp conditions, based on the Assouad (co)dimension of $E$, for the inclusion of $w$ in Muckenhoupt's $A_p$ classes of weights, $1\le p<\infty$. With the help of general $A_p$-weighted embedding results, we then prove (global) Hardy--Sobolev inequalities and also fractional versions of such inequalities in the setting of metric spaces.
\end{abstract}

\date{\today}
\maketitle

\section{Introduction}

Muckenhoupt's $A_p$-weights are important tools in mathematical analysis. Their characterizing property in $\R^n$ is that the Hardy--Littlewood maximal operator is bounded in the weighted space $L^p(w\,dx)$, for $1< p <\infty$, if and only if the weight $w$ is an $A_p$-weight; see e.g.~\cite{garcia-cuerva}. Muckenhoupt's $A_p$-weights are also examples of {\em $p$-admissible} weights in the sense of~\cite{HKM}, and hence $A_p$-weighted Euclidean spaces satisfy the basic assumptions that are often used in the theory of analysis on metric spaces: the doubling property and a $p$-Poincar\'e inequality. Due to the importance of $A_p$-weighted function spaces, various norm inequalities have been established for $A_p$-weights both in Euclidean spaces and in more general settings; see, for instance~\cite{garcia-cuerva,MW,MR1052009,MR1962949}.

On the other hand, Hardy and Hardy--Sobolev -type inequalities are important examples of inequalities that in particular yield embeddings between weighted function spaces. When $X$ is a metric space, the weights in these inequalities are of the type $\delta_E^{-\alpha}$, $\alpha\in\R$, where $\delta_E(x)=\dist(x,E)$ denotes the distance from a point $x\in X$ to a closed set $E\subset X$. We refer to~\cite{LV} and~\cite{LHA} for recent results related to Hardy--Sobolev -inequalities in $\R^n$ and Hardy inequalities in metric spaces, respectively, and to~\cite{MR3148524,MR3343059} and~\cite{MR3237044} for fractional counterparts of such inequalities, respectively.

Now, two natural questions arise: 
\begin{itemize}
\item[(i)] When does such a weight $\delta_E^{-\alpha}$ belong to (some) Muckenhoupt $A_p$-class? 
\item[(ii)] Can the general theory of $A_p$-weighted inequalities be used to derive certain Hardy and Hardy--Sobolev -inequalities?
\end{itemize}
As far as we know, neither of these questions has been given a complete or comprehensive answer, although there are several partial results concerning question (i), mainly in $\R^n$ but also in more general metric spaces; we will comment on some of these more precisely in Section~\ref{s.distance}.
On the other hand, some of the results of Horiuchi~\cite{MR1021144,MR1118940} are closely related to question~(ii) in $\R^n$.

In this paper, we provide a characterizing answer to question (i) with high generality. In an Ahlfors $Q$-regular metric space $X$ (and hence in particular in $\R^n$ with $Q=n$) our result reads as follows:

\begin{thm}\label{thm.char_intro}
Assume that $X$ is a $Q$-regular metric space.
Let $\emptyset\neq E\subset X$ be a closed set with $\dima(E)<Q$ and 
let $\alpha\in\R$ and $w=\delta_E^{-\alpha}$. 
Then 
\begin{itemize}
\item[(A)] $w\in A_p$, for $1<p<\infty$, if and only if $(1-p)(Q-\dima(E)) < \alpha < Q-\dima(E)$\,.
\item[(B)] $w\in A_1$ if and only if $0 \le \alpha < Q-\dima(E)$\,.
\end{itemize}
\end{thm}

Here $\dima(E)$ is the (upper) Assouad dimension of the set $E$. In the more general setting of a non-Ahlfors-regular space $X$ we obtain in Corollary~\ref{c.porous} a corresponding result in terms of the (lower) Assouad codimension. In particular, these characterizations show that the Assoaud dimension and codimension are certainly the correct notions of dimension to consider in this context. The definitions of these dimensions and other relevant concepts, as well as other preliminaries such as our assumptions on the metric space $X$, will be reviewed in Section~\ref{s.metric}.  Section~\ref{s.distance} then contains all of our results related to the $A_p$-properties of the distance functions.   

Concerning question (ii), recall that in $\R^n$ the global Hardy--Sobolev inequality, for exponents $1\le p\le q\le np/(n-p)<\infty$ and $\beta\in\R$ and with respect to a closed set $E\subset\R^n$, reads as
\begin{equation}\label{e.eucl.weighted.intro}
\biggl(\int_{\R^n} \lvert f(x)\rvert^q \, \delta_E(x)^{(q/p)(n-p+\beta)-n}\,dx\biggr)^{1/q}
\le C\biggl(\int_{\R^n} \lvert \nabla f(x)\rvert^p\,\delta_E(x)^\beta\,dx\biggr)^{1/p}.
\end{equation}
When the closed set $E\subset\R^n$ is given, the main question is whether there exists a constant $C>0$ such that this inequality 
holds for every $f\in C^\infty_0(\R^n)$. Notice that inequality~\eqref{e.eucl.weighted.intro} includes many well-known special cases: for $1\le p<n$, $q=np/(n-p)$ and $\beta=0$ we recover the usual Sobolev inequality, 
the case $p=q$ is the (weighted) $(p,\beta)$-Hardy inequality, and for $E=\{0\}$ these inequalities are known as Caffarelli--Kohn--Nirenberg inequalities. 

We consider metric space versions of Hardy--Sobolev inequalities in Section~\ref{s.application}. Natural test functions in this setting are Lipschitz functions with bounded support, and then $|\nabla f|$ is replaced with an upper gradient of $f$.  
In a $Q$-regular metric space $X$, the exponent $Q>1$ plays the same role as $n$ does in $\R^n$, and in this case the closed set $E\subset X$ is assumed to satisfy 
\begin{equation}\label{e.Q-reg.bound}
\dima(E) < \min \biggl\{ \frac q p (Q-p+\beta) \, , \, Q - \frac{\beta}{p-1}  \biggr\}\,\,.
\end{equation}
In a non-regular space $X$, the Assouad dimension is again replaced with the codimension, and $Q>1$ is assumed to be such that the measure lower bound $\mu(B(x,r))\ge Cr^Q$ holds for all $x\in X$ and all $r>0$.

The proofs of these Hardy--Sobolev inequalities are based on the knowledge of the $A_p$-properties of the distance weights and the general theory of $A_p$-weighted inequalities that has been developed in the setting of metric spaces by P\'erez and Wheeden in~\cite{MR1962949} and that will be discussed with more details in Section~\ref{s.weights}. Our main tool concerning general $A_p$-weighted theory is Theorem~\ref{t.r_bounded}, which is essentially a combination of~\cite[Theorem~2.1]{MR1962949} and~\cite[Theorem~2.4]{MR1962949}. The first of these two is, in turn, a metric space generalization of an Euclidean result due to Muckenhoupt and Wheeden~\cite{MW} giving single weight control for the Riesz potential $\mathcal{I}_s$ in terms of a maximal operator, while the second result gives a two weight $L^p$--$L^q$ control for this maximal operator; a Euclidean version of the latter was proved by P\'erez \cite[Theorem~1.1]{MR1052009}.

In addition, we establish in Section~\ref{s.fractional} fractional versions of Hardy--Sobolev inequalities in metric spaces, again based on the general theory from Section~\ref{s.weights}. In the fractional case, with order of smootheness $0<s<1$, the first term in the minimum in~\eqref{e.Q-reg.bound} is replaced with $\frac q p (Q-sp+\beta)$ and the measure lower bound $\mu(B(x,r))\ge Cr^Q$ is assumed to hold with $Q>s$. We also show the optimality of the bound $\dima(E) < \frac q p (Q-sp+\beta)$ for fractional inequalities; see Proposition~\ref{p.weighted_Fract_optimality}. In $\R^n$, the optimality of the corresponding bound is well-understood for non-fractional inequalities, cf.~\cite{LV}. On the contrary, we do not know whether the bound $\dima(E) < Q - \frac{\beta}{p-1}$ is really needed in the case of fractional inequalities, but since this bound can be seen to be necessary at least in some instances of the non-fractional inequalities, it is not possible to remove this bound from the general $A_p$-approach. See Remarks~\ref{rem.optimality} and~\ref{rem.opt_HS} for more discussion on these dimensional bounds. 

In the special case $X=\R^n$, our results concerning the Hardy--Sobolev inequality~\eqref{e.eucl.weighted.intro} are essentially the same as the corresponding results from~\cite{LV} when $\beta\le 0$. On the other hand, when $\beta>0$ and $\dima(E)<n-1$, our results are weaker than the results in~\cite{LV}, but when $\beta>0$ and $\dima(E)\ge n-1$, we actually obtain an improvement to the results in~\cite{LV}; see Remark~\ref{rem.impro}. In more general metric spaces, all of our results concerning Hardy--Sobolev inequalities and their fractional versions appear to be new in the case $q>p$.

Let us also remark that in this work we only consider {\em global} Hardy--Sobolev inequalities, that is, the integrations are taken over the whole space $X$. In these inequalities the set $E$ needs to be rather ``thin'', which is illustrated by the fact that there is an upper bound for the dimension of $E$. In addition, in this case the test functions need not vanish in $E$. In the other typical instance of Hardy--Sobolev inequalities one assumes that $E$ is ``thick'',  whence there in particular is a lower bound for the dimension of $E$, and then the test functions are assumed to have a compact support in the open set $\Omega=X\setminus E$. See for instance~\cite{LV} for the Euclidean case of these inequalities and more comments related to the distinction between the ``thin'' and ``thick'' cases. Our characterization of the distance-type $A_p$-weights in fact shows that the present approach to Hardy--Sobolev -inequalities can not be applied in the ``thick'' case, where the natural dimensional bounds for the inequalities are not compatible with the bounds for the $A_p$-properties of distance functions.

In conclusion, the present $A_p$-weight approach to Hardy--Sobolev inequalities certainly has limitations, especially when dealing with weight exponents $\beta>0$, and it can not be applied in the ``thick'' case of these inequalities, where different tools need to be used. Nevertheless, as a positive answer to question (ii) we see that in many cases --- in particular in the most important case $\beta=0$ of the global Hardy--Sobolev -inequalities --- the present approach indeed yields optimal results for global Hardy--Sobolev -inequalities and also for the corresponding fractional inequalities both in $\R^n$ and in more general metric spaces. 

\section{Preliminaries on metric spaces}\label{s.metric}

We assume throughout this paper that $X=(X,d,\mu)$ is 
a metric measure space (with $\#X\ge 2$), where
$\mu$ is a Borel measure supported on $X$ such that $0<\mu(B)<\infty$ for
all (open) balls 
\[B=B(x,r):= \{y\in X : d(x,y)< r\}\]
with $x\in X$ and $r>0$. We make the tacit assumption that each ball $B\sub X$ has a fixed center $x_B$ 
and radius $\rad(B)$, and thus notation such as $\ell B = B(x_B,\ell \rad(B))$ 
is well-defined for all $\ell>0$. When $E,F\subset X$, we let $\diam(E)$ denote the diameter of $E$ and $\dist(E,F)$ is the distance between the sets $E,F\subset X$, and in particular we use $\delta_E(x)=\dist(x,E)=\dist(\{x\},E)$ to denote the distance from a point $x\in X$ to the set $E$.

We also assume throughout
that $\mu$ is \emph{doubling}, that is, there is a constant $C_D>0$ such that 
whenever $x\in X$ and $r>0$, we have
\begin{equation}\label{e.doubling}
  \mu(B(x,2r))\le C_D\, \mu(B(x,r)).
\end{equation}
Iteration of~\eqref{e.doubling} shows that if $\mu$ is doubling, then there exist an exponent $\sigma>0$ and a constant $C_*>0$ such that the quantitative doubling condition
\begin{equation}\label{e.doubling_quant}
  \frac{\mu(B(y,r))}{\mu(B(x,R))}\ge C_*\Bigl(\frac {r}{R}\Bigr)^\sigma
\end{equation}
holds whenever $B(y,r)\subset B(x,R)\subset X$; see~\cite[Lemma~3.3]{BB}.

In some of our results we also need to assume that
for a given exponent $\eta>0$ there 
is a constant $C^*=C^*(X,\eta)>0$ 
such that the {\em (quantitative) reverse doubling condition} 
\begin{equation}\label{reverse doubling}
 \frac{\mu(B(y,r))}{\mu(B(x,R))} \le C^*\Bigl(\frac{r}{R}\Bigr)^\eta
\end{equation}
holds whenever 
$B(y,r)\subset B(x,R)\subset X$. Notice that under this condition
$\mu(\{x\})=0$ for all $x\in X$ and 
the space $X$ is necessarily unbounded, since estimate~\eqref{reverse doubling}
holds for arbitrary large radii $R$.
If the space $X$ is unbounded and connected
(and $\mu$ is doubling, as we assume throughout), 
then there exists some $\eta>0$ such that \eqref{reverse doubling}
holds whenever $B(y,r)\subset B(x,R)\subset X$; cf.~\cite[Corollary~3.8]{BB}.

The space  $X=(X,d,\mu)$ is said to be {\em Ahlfors $Q$-regular} (or simply {\em $Q$-regular}), for $Q > 0$, 
if there is a constant $C \ge 1$ such that
\[
  C^{-1}r^Q \le \mu(B(x,r)) \le Cr^Q
\]
for all $x \in X$ and every $0 < r < \diam(X)$.
Notice that if $X$ is $Q$-regular, then $\mu$ is doubling,
and moreover~\eqref{reverse doubling} holds for all $\eta \le Q$ if $X$ is unbounded. 
The Ahlfors $Q$-regularity of the space $X$ is equivalent to the requirement that
there is a constant $C \ge 1$ such that
\[
  C^{-1}r^Q \le \Ha^Q(B(x,r)) \le Cr^Q
\]
for all $x \in X$ and every $0 < r < \diam(X)$, where $\Ha^Q$ is the $Q$-dimensional
Hausdorff measure on $X$. 
Consult, for instance, \cite[Section~1.4]{MackayTyson2010}
for the definition of the Hausdorff measure and the above equivalence
concerning $Q$-regularity.

If the space $X$ is not Alhfors $Q$-regular, then
it is often convenient to describe the sizes of the subsets of $X$ in terms of 
\emph{codimensions} rather than dimensions. For instance, the 
\emph{Hausdorff codimension} of $E\sub X$ 
(with respect to $\mu$) is the number
\[\codimh(E) = \sup\big\{q\geq 0 : \Ha_R^{\mu,q}(E)=0\big\},\]
where 
\[
\Ha_R^{\mu,q}(E) = \inf\bigg\{ \sum_{k} \rad(B_k)^{-q}\mu(B_k) : E \subset \bigcup_{k} B_k,\ 
                               \rad(B_k) \le R \bigg\}
\] 
is the \emph{Hausdorff content of codimension $q$};
if $\mu(E)>0$, then we set $\codimh(E)=0$.
If $X$ is $Q$-regular, 
then we have for all $E\sub X$ that $Q-\codimh(E)=\dimh(E)$, the usual Hausdorff dimension.

For this paper, the most important notion of (co)dimension is the Assouad (co)dimension.
When $E\sub X$, the \emph{(upper) Assouad dimension} of $E$, denoted $\udima(E)$ (or simply $\dima(E)$, as in the Introduction), 
is the infimum of exponents $s\ge 0$ for which 
there is a constant $C \ge 1$ such that for all $x\in E$ and every $0<r<R<2\diam(X)$, the set $E\cap B(x,R)$ can be covered by at most $C(r/R)^{-s}$ balls of radius $r$. 
We remark that $\udima$ is the ``usual'' Assouad dimension
found in the literature, and we refer to~\cite{MR1608518} for its basic properties and a 
historical account and to~\cite{Fraser} (and the references therein) for more recent
results related to the (upper) Assouad dimension.

The corresponding codimension, the \emph{(lower) Assouad codimension} $\lcodima(E)$,
is defined in terms of the measures of the
(open) $r$-neighborhoods
\[E_r=\{x\in X:\dist(x,E)<r\}\]
of $E\sub X$.
Namely, $\lcodima(E)$ is the supremum of
all $\rho \ge 0$ for which there exists a constant $C \ge 1$ such that
\begin{equation*}\label{eq:bouli*}
\frac{\mu(E_r\cap B(x,R))}{\mu(B(x,R))}\le C\Bigl(\frac r R\Bigr)^\rho
\end{equation*}
for every $x\in E$ and all $0<r<R<2\diam(X)$.
Notice in particular that $\lcodima(E)>0$ implies that $\mu(E)=0$,
by the Lebesgue differentiation theorem; see e.g.~\cite[Theorem~1.8]{HEI}.
If $X$ is $Q$-regular, then it is not hard to see that
$\udima(E)  = Q - \lcodima(E)$ 
for all $E\sub X$, cf.~\cite[(3.11)]{KLV}. On the other hand, if $E\subset X$
is Ahlfors $\lambda$-regular (as a subspace of $X$, endowed with the induced metric and the $\lambda$-dimensional Hausdorff measure $\Ha^\lambda$), then $\dimh(E)= \udima(E)=\lambda$;
see e.g.~\cite[Section~1.4.4]{MackayTyson2010}.

Let us remark here that the terminology of upper Assouad dimension and
lower Assouad codimension is due to the fact that there also exists a corresponding
``dual'' pair, i.e., the lower Assouad dimension and
the  upper Assouad codimension; see~\cite{KLV}. Neither of these two will be needed in this paper,
but they play an important role in the ``thick'' cases of Hardy and Hardy--Sobolev
inequalities; see~\cite{LHA,LV}.

When $E\subset X$, a function $u\colon E\to \R$ is said to be ($L$-)\emph{Lipschitz}, if
\[
|u(x)-u(y)|\leq L d(x,y)\qquad \text{ for all } x,y\in E\,.
\]
We denote the set of all Lipschitz functions $u\colon E\to\R$ by $\Lip(E)$. 
In addition, $\Lip_0(X)\subset \Lip(X)$ denotes the set of all Lipschitz functions 
$u\in\Lip(X)$ for which there exists some ball $B$ such that $u(x)=0$ for all $x\in X\setminus B$.

When $B$ is a ball in $X$, the integral average of a function
$u\in L^1(B)$ is
\[
   u_B:=\frac{1}{\mu(B)}\, \int_B u\, d\mu =:\vint_{B}\, u\, d\mu\,.
\]

A {\em weight} is a measurable function $w$ on $X$ such that 
$w(x)>0$ for $\mu$-almost every $x\in X$
and $\int_B w\,d\mu<\infty$ whenever $B\subset X$ is a ball. We write 
$w(E)=\int_E w\,d\mu$ if $E\subset X$ is a measurable set and
$w$ is a weight.
A weight $w$ belongs to the {\em Muckenhoupt class $A_p$} of weights (i.e. $w\in A_p$), for $1\le p<\infty$, if
there is a constant $A>0$ such that, for every ball $B$ in $X$, 
\begin{equation}\label{a_p}
\biggl( \vint_B w\,d\mu\biggr) \biggl(\vint_B w^{-1/(p-1)}\,d\mu\biggr)^{p-1} \le A\,
\quad\text{ if } p>1\,,
\end{equation}
and
\begin{equation}\label{a_1}
\biggl( \vint_B w\,d\mu\biggr) \esssup_{y\in B} \frac{1}{w(y)} \le A\,
\quad\text{ if } p=1\,.
\end{equation}
It follows from these $A_p$ conditions
and H\"older's inequality that Muckenhoupt weights  
satisfy the following strong doubling property:  if $1\le p<\infty$ and $w\in A_p$, then
\begin{equation}\label{e.sdoubling}
w(B)\le A\bigg(\frac{\mu(B)}{\mu(E)}\bigg)^p w(E)
\end{equation}
whenever $E$ is a measurable subset of a ball $B\subset X$ with $\mu(E)>0$. 
In particular, the measure $w\,d\mu$ satisfies
the doubling condition \eqref{e.doubling}. 
When $1<p<\infty$, it follows immediately from the $A_p$-condition~\eqref{a_p} for a weight $w$ that  
\begin{equation}\label{e.Ap_equiv}
 w\in A_p \iff w^{1/(1-p)}\in A_{p/(p-1)}\,.
\end{equation}

A weight $w$ is said to belong to the {\em Muckenhoupt class $A_\infty$} (i.e. $w\in A_\infty$) 
if there are constants $C>0$ and $\delta>0$ such that
\[
w(E) \le C \biggl(\frac{\mu(E)}{\mu(B)}\biggr)^\delta w(B)
\]
whenever $E$ is a measurable subset of a ball $B\subset X$. 
By \cite[Chapter I, Theorem 15]{StrombergTorchinsky}, it holds for 
every $1<p<q<\infty$ that
\begin{equation}\label{ap_relations}
A_1\subset A_p\subset A_q\subset A_\infty.
\end{equation}
It is also well known that in the Euclidean case with the Lebesgue measure
$A_\infty=\bigcup_{1\le p<\infty} A_p$.
In a metric space $X$ this equality is valid under the assumptions that
the measure $\mu$ is doubling and
$\mu(B(x,r))$ increases
continuously with $r$ for each $x\in X$;
we refer to~\cite[Chapter I, Theorem~18]{StrombergTorchinsky}. 
However, there
exist metric spaces where the class of $A_\infty$-weights is strictly larger than the
union $\bigcup_{1\le p<\infty} A_p$; see~\cite{MR2815740} and~\cite{StrombergTorchinsky}.

\section{Powers of distance functions as weights}\label{s.distance}

In this section we investigate the connections between 
the (lower) Assouad codimension of a closed set $E\subset X$
and the $A_p$-properties of the powers of the distance function 
$\delta_E = \dist(\cdot,E)$. Recall that $X$ is a metric space equipped with a doubling mesure $\mu$;
no further assumptions on $X$ are needed in this section.

\begin{defn}
We say that a closed set $\emptyset\neq E\subset X$ satisfies the {\em Aikawa condition} with an exponent $\alpha\geq 0$ and a constant $C \ge 1$ if inequality
\begin{equation}\label{eq:aikawa}
  \int_{B(x,r)} \delta_E(y)^{-\alpha}\,d\mu(y) \le C r^{-\alpha}\mu(B(x,r))
\end{equation}
holds for every $x\in E$ and all $0<r<2\diam(X)$.
We interpret the integral to be $+\infty$ if $\alpha>0$ and $E$ has a positive measure.
\end{defn}

\begin{rem}\label{rmk:LT}
Let $\emptyset\neq E\subset X$ be a closed set.
The lower Assouad codimension of $E$ can be characterized as
the supremum of all exponents $\alpha\ge 0$ for which 
$E$ satisfies the Aikawa condition with some constant $C\ge 1$.
In particular, the Aikawa condition holds for all  
$\alpha<\lcodima(E)$ (for $\alpha\le 0$ this is trivial).
This characterization is essentially \cite[Theorem~5.1]{MR3055588}.  Notice
that in~\cite{MR3055588} the relevant radii are always bounded from above by
$\diam(E)$, whereas presently the upper bound for radii is $2\diam(X)$ both in the
definition of $\lcodima$ and in the above Aikawa condition~\eqref{eq:aikawa}. Nevertheless, the proof from~\cite{MR3055588} works also in this case with obvious minor
modifications. 

We also remark that if $\alpha<\lcodima(E)$ and $\mu(E)=0$, it follows from~\eqref{eq:aikawa} 
that the function $w=\delta_E^{-\alpha}=\dist(\cdot,E)^{-\alpha}$ is a weight. Since
$\lcodima(E)>0$ implies that $\mu(E)=0$, it in particular follows that the function 
$w$ is a weight if $0\le \alpha<\lcodima(E)$.
\end{rem}

A concept of dimension defined via integrals as in~\eqref{eq:aikawa}
was used by Aikawa in~\cite{Aikawa1991} for subsets of $\R^n$.
Thus, e.g.\ in~\cite{MR3055588}, where the interest originated from such
integral estimates, the lower Assouad codimension was called
the \emph{Aikawa codimension}.

The following lemma, essentially \cite[Lemma~2.2]{LHA}, records the fact that the Aikawa condition~\eqref{eq:aikawa} enjoys
self-improvement. This result is a straight-forward consequence of the famous 
self-improvement result for reverse H\"older inequalities, which in $\R^n$ is due to
Gehring~\cite{geh}. The proof in~\cite{LHA} is based on a metric space version of
the Gehring lemma; see e.g.~\cite[Theorem~3.22]{BB}.
We emphasize that besides the doubling property of $\mu$
no other assumptions are required for the space $X$ in Lemma~\ref{lemma:aikawa si}.

\begin{lem}\label{lemma:aikawa si}
Let $\emptyset\neq E\sub X$ be a closed set that
satisfies the Aikawa condition
with an exponent $\alpha >0$
and a constant $C_0\ge 1$. 
Then there exist $\delta>0$ and $C\ge 1$, depending only on the 
given data, such that $E$ satisfies the Aikawa condition 
with the exponent $\alpha+\delta$
and the constant $C$. 
\end{lem}

The following theorem is our main result concerning the $A_p$-properties of
distance weights.
In~\cite{MR3215609}, corresponding results were obtained in metric spaces,
but using a completely different approach and
under the much stronger assumption that both $X$ and $E$ satisfy 
Ahlfors regularity conditions; see e.g.~\cite[Theorems~6 and~7]{MR3215609}.

\begin{thm}\label{t.a_infty}
Let $\emptyset\neq E\subset X$ be a closed set and let $\alpha\in\R$ and
$w=\delta_E^{-\alpha}$. 
Then the following statements hold.
\begin{itemize}
\item[(A)] If $\lcodima(E)>\alpha\geq 0$, then $w\in A_p$ for all $1 \le p \le\infty$.
\item[(B)] If $\alpha<0$ and $1<p<\infty$ are such that
\[
\lcodima(E)>\frac{\alpha}{1-p}\,, 
\]
then $w\in A_p$.
\item[(C)] If $\lcodima(E)>\max\{0,\alpha\}$, then $w\in \bigcup_{1\le p<\infty} A_p$.
\end{itemize}
\end{thm}

\begin{proof}
(A) Since $A_1\subset A_p$ for all $p\ge 1$, it suffices to show that $w\in A_1$,
i.e., that inequality~\eqref{a_1} holds for all balls in $X$. 

To this end, fix a ball $B=B(x,r)$ in $X$; without loss
of generality, we may assume that $0<r<2\diam(X)$. Assume first that $2B\cap E \not=\emptyset$. 
Then there is $z\in E$ such that $B\subset B(z,3r)$. 
By Remark \ref{rmk:LT}, inequality \eqref{eq:aikawa} holds for the exponent $\alpha$. 
Using this inequality and the doubling property of $\mu$, we obtain
\begin{equation}\label{e.first}
\begin{split}
\vint_{B(x,r)} w(y)\,d\mu(y)
&\le \frac{1}{\mu(B(x,r))}\int_{B( z,3r)} \delta_E(y)^{-\alpha}\,d\mu(y)\\
&\le Cr^{-\alpha}\frac{\mu(B( z,3r))}{\mu(B(x,r))}
\le Cr^{-\alpha}\,.
\end{split}
\end{equation}
Since $\alpha\ge 0$, we have for every $y\in B\setminus E$ that
\[
\frac{1}{w(y)} = \delta_E(y)^{\alpha} \le d(y, z)^{\alpha}\le 3^{\alpha} r^{\alpha}\,.
\]
Combining~\eqref{e.first} and the above estimate and using the fact that $\mu(E)=0$,
we see that inequality~\eqref{a_1} holds in the case $2B\cap E\not=\emptyset$.

On the other hand, if $2B\cap E=\emptyset$, then 
\begin{equation}\label{distances}
\delta_E(y)/3\le \dist(B,E)\le \delta_E(y) 
\end{equation}
for all $y\in B$.
It follows that 
\[
\biggl(\vint_{B(x,r)} w(y)\,d\mu(y)\biggr)\esssup_{y\in B} \frac{1}{w(y)} \le C\dist(B,E)^{-\alpha}\dist(B,E)^{\alpha}\le C\,,
\]
and thus inequality~\eqref{a_1} holds also in the case $2B\cap E=\emptyset$. 
This proves that $w\in A_1$, as desired.

(B)
From (A) it follows that $\delta_E^{-\alpha/(1-p)} \in A_{\tilde{p}}$ for all $1\leq \tilde{p}\leq \infty$. In particular, \[\delta_E^{-\alpha/(1-p)} \in A_{p/(p-1)}\,,\] which by \eqref{e.Ap_equiv} implies that $w=\delta_E^{-\alpha} \in A_p$.

The final statement (C) follows from a combination of parts (A) and (B).
\end{proof}

Next we turn to partial converses of the statements (A) and (B) in Theorem~\ref{t.a_infty}. 
The following Theorem~\ref{e.self} reveals a surprising self-improvement phenomenon for the 
$A_p$-properties of functions  
$\delta_E^{-\alpha}$, where $\alpha>0$ and $E$ is porous. 
Recall that a set $E\subset X$ is porous, if there is a constant $0 < c < 1$ such that for every 
$x\in E$ and all $0 <r<2\diam(X)$ there exists a point $y\in X$ such that $B(y,cr)\subset B(x,r)\setminus E$.

\begin{thm}\label{e.self}
Let $\emptyset\neq E\sub X$ be a closed and
porous set, and let $\alpha>0$ and $w=\delta_E^{-\alpha}$. Then the following conditions are equivalent.
\begin{itemize}
\item[(A)] $\lcodima(E)>\alpha$\,;
\item[(B)] $w\in A_1$\,;
\item[(C)] $w \in A_q$, for some $1<q<\infty$.
\end{itemize}
\end{thm}

\begin{proof}
By Theorem~\ref{t.a_infty}, condition (A) implies both conditions (B) and (C), and
from the inclusions in~\eqref{ap_relations} it follows that (B) implies (C). Thus it suffices to show that (C) implies (A).

Let us hence assume that $w\in A_q$ for some $1<q<\infty$. 
By the $A_q$-condition~\eqref{a_p}, there is a constant $C>0$ such that 
\[
\biggl( \vint_B w(y)\,d\mu(y)\biggr) \biggl(\vint_B w(y)^{-1/(q-1)}\,d\mu(y)\biggr)^{q-1} \le C
\]
for every ball $B$ in $X$.
Now fix $x_0\in E$ and $0<r_0<2\,\diam(X)$, and let $B=B(x_0,r_0)$. 
Since $E$ is porous, there is a ball $B(x,r)=B'\subset B$ such that $r=cr_0/2$ and $\dist(B',E)\ge r/2$, and thus, using also the doubling property~\eqref{e.doubling_quant}, we obtain 
\begin{align*}
\biggl(\vint_B w(y)^{-1/(q-1)}\,d\mu(y)\biggr)^{q-1} 
&\ge \biggl(\frac{\mu(B(x,r))}{\mu (B(x_0,r_0))}\biggr)^{q-1}\biggl(\vint_{B'} w(y)^{-1/(q-1)}\,d\mu(y)\biggr)^{q-1} \\
&\ge  C\biggl(\vint_{B'} r^{\alpha/(q-1)}\,dy\biggr)^{q-1}\ge C r^{\alpha}.
\end{align*}
Combining the previous two estimates, we obtain
\begin{align*}
\int_{B} \delta_E(y)^{-\alpha}\,d\mu(y) = \int_B w(y)\,d\mu(y)
&\le C \mu(B)\biggl(\vint_B w(y)^{-1/(q-1)}\,d\mu(y)\biggr)^{1-q}\\
&\le C \mu(B)r^{-\alpha} \le C\mu(B)r_0^{-\alpha},
\end{align*}
showing that the closed set $E$ satisfies the Aikawa condition~\eqref{eq:aikawa} with the exponent
$\alpha>0$. By the self-improvement of the Aikawa condition, Lemma~\ref{lemma:aikawa si},
there then exists $\delta>0$ such that the Aikawa condition holds also with
the exponent $\alpha+\delta$, and so it follows from Remark~\ref{rmk:LT} that 
$\lcodima(E)\ge \alpha+\delta>\alpha$, proving condition (A).
\end{proof}

\begin{rem}\label{r.ap}
To see that a converse to Theorem~\ref{t.a_infty}(B) holds,
when $\emptyset\neq E\sub X$ is a closed and porous set,
we assume that $\alpha<0$ and $1<p<\infty$ are such that $w=\delta_E^{-\alpha}\in A_p$.
By the equivalence in~\eqref{e.Ap_equiv}, we then have that 
\[
\delta_E^{-\alpha/(1-p)} = w^{1/(1-p)}\in A_{p/(p-1)}\,, 
\]
and so it follows from Theorem~\ref{e.self} that $\lcodima(E)>{\alpha}/(1-p)$. 
\end{rem}

\begin{cor}\label{c.porous}
Let $\emptyset\neq E\subset X$ be a closed and porous set, and 
let $\alpha\in\R$ and $w=\delta_E^{-\alpha}$. 
Then 
\begin{itemize}
\item[(A)] $w\in A_p$, for $1<p<\infty$, if and only if $(1-p)\lcodima(E) < \alpha < \lcodima(E)$\,.
\item[(B)] $w\in A_1$ if and only if $0 \le \alpha < \lcodima(E)$\,.
\end{itemize}
\end{cor}

\begin{proof} 
All the claims, except the necessity of the condition $\alpha\ge 0$ in part (B), follow from Theorems~\ref{t.a_infty} and~\ref{e.self}, Remark~\ref{r.ap}, and  
the fact that $\lcodima(E)>0$ for porous sets (cf.~\cite[Remark~3.6]{KLV}).
The remaining claim can be justified as follows.
If $\alpha<0$,  then for all balls $B=B(x,r)$ with $x\in E$
it holds that  $\vint_B w\,d\mu\ge cr^{-\alpha}$, with $c>0$ independent of $x$ and $r$, while 
 \[\esssup_{y\in B} \frac{1}{w(y)}=\infty\,.\] Hence the $A_1$ condition~\eqref{a_1} is clearly not satisfied,
and so we conclude that $w\in A_1$ is possible only when $\alpha\ge 0$. 
\end{proof}

If $X$ is $Q$-regular, then 
$\lcodima(E)=Q-\udima(E)$ for all $E\subset X$, 
and moreover a set $E\subset X$ is porous
if and only if $\udima(E)<Q$; see e.g.~\cite[Lemma~3.12]{BHR}.
Hence Theorem~\ref{thm.char_intro} follows immediately from Corollary~\ref{c.porous}. 
In particular, in the Euclidean space $\R^n$ with the Lebesgue measure, we obtain the following Corollary~\ref{c.a_infty_rn}. 
Similar characterizations were obtained already by Horiuchi in~\cite[Lemma~2.2]{MR1118940}.
However, in~\cite{MR1118940} the dimensional condition for $E\subset\R^n$ was formulated using the so-called
``$P(s)$-property''. This property was only recently shown to be intimately connected with the
Assouad dimension, see~\cite[Theorem~3.4]{LV}.

\begin{cor}\label{c.a_infty_rn}
Let $\emptyset\neq E\subset \R^n$ be a closed set with $\udima(E)<n$, and let $\alpha\in\R$ and
$w=\delta_E^{-\alpha}$.
Then 
\begin{itemize}
\item[(A)] $w\in A_p$, for $1<p<\infty$, if and only if $(1-p)(n-\udima(E)) < \alpha < n-\udima(E)$\,.
\item[(B)] $w\in A_1$ if and only if $0\le \alpha < n-\udima(E)$\,.
\end{itemize}
\end{cor}

Corollary~\ref{c.a_infty_rn} is also closely related to~\cite[Lemma~3.3]{MR2606245},
which states that if a compact set $\emptyset\not=E\subset \R^n$ is 
a subset of an Ahlfors $\lambda$-regular set, for $0\le \lambda < n$,
and if $1<p<\infty$ and $\alpha\in\R$ are such that
\[
(1-p)(n-\lambda) < \alpha < n-\lambda\,,
\]
then $w=\delta_E^{-\alpha}$ is an $A_p$ weight in $\R^n$.
Recall that a compact set $\emptyset\not=F\subset \R^n$
is Ahlfors $\lambda$-regular if there
is $C\ge 1$ such that
\[
C^{-1} r^{\lambda} \le \Ha^{\lambda}(F\cap B(x,r))\le Cr^{\lambda}
\]
for each $x\in F$ and all $0<r\le \diam(F)$
(or for all $r>0$ if $F$ consists of a single point),
and that then $\udima(F)=\dimh(F)=\lambda$.
In particular, if a closed set $E\subset\R^n$ is a subset of an Ahlfors $\lambda$-regular set, then $\udima(E)\le \lambda$.

Finally, let us note that
Corollary~\ref{c.a_infty_rn} naturally contains the well known results 
for the particular case $E=\{0\}\subset\R^n$, in which $\udima(E)=0$.  
Indeed, let $w(x)=\lvert x\rvert^{-\alpha}$ for $x\in\R^n$.
Then it follows from Corollary~\ref{c.a_infty_rn} that $w\in A_1$ if and only if $0\le \alpha < n$,
and $w\in A_p$, for $1<p<\infty$, if and only if $(1-p)n < \alpha < n$.
These same bounds can be found, for instance, in~\cite[p.~229,~p.~236]{MR869816}.

\section{Boundedness results for Riesz potentials}\label{s.weights}

{\emph{Throughout the remainder of this paper, the following assumptions
are maintained:
\begin{itemize}
\item[(S1)] $X$ is an unbounded metric space,
equipped with a doubling measure $\mu$ such that 
\[\mu(\{x\})=0\,,\qquad \text{ for every }x\in X\,.\]
\item[(S2)] The annulus $B(x,R)\setminus B(x,r)$ is non-empty
for each $x\in X$ and every $0<r<R<\infty$.
\end{itemize}}
We remark that (S1) follows from 
the reverse doubling condition \eqref{reverse doubling}
with any $\eta>0$.
On the other hand, from condition (S2) it follows that 
the radius $\mathrm{rad}(B)=r$ of a ball $B=B(x,r)$ in $X$ is uniquely determined and
inequality $\rad(B_1)\le 2\rad(B_2)$ holds for all balls $B_1\subset B_2\subset X$.
\medskip

When $s>0$, the Riesz potential $\mathcal{I}_sf=\mathcal{I}_s(f)$ of a measurable function $f\ge 0$ is defined by
\begin{equation}\label{d.potential}
\mathcal{I}_s(f)(x) = \int_{X} \frac{f(y)d(x,y)^s}{\mu(B(x,d(x,y)))} \,d\mu(y)\,,\quad x\in X\,.
\end{equation}
Since $\mu(\{x\})=0$ for each $x\in X$, we can tacitly restrict the above integration
to the set $X\setminus \{x\}$ in order to avoid difficulties when $x=y$.

The following theorem gives a sufficient condition for the validity of certain
two weight inequalities for the Riesz potentials, where the weights are powers of
a distance function.

\begin{thm}\label{t.riesz}
Let $s>0$. Assume that 
the reverse doubling condition \eqref{reverse doubling} holds
with the exponent $\eta= s$ and that there is $Q>s$ such that
$\mu(B)\ge c\rad(B)^Q$ for all balls $B\subset X$.

Let $\emptyset \neq E\subset X$ be a closed set, and let  
$1<p\le q\le Qp/(Q-sp)<\infty$ and $\beta\in\R$ be such that
\begin{equation}\label{e.riesz_assumption}
\lcodima(E) > \max \biggl\{ Q - \frac q p (Q-sp+\beta) \, , \, \frac{\beta}{p-1}  \biggr\}\,.
\end{equation}
Then there is a constant $C>0$ such that inequality 
\begin{equation}\label{e.r_weighted}
\biggl(\int_{X} \mathcal{I}_s (f)(x)^q \, \delta_E(x)^{(q/p)(Q-sp+\beta)-Q}\,d\mu(x)\biggr)^{1/q}
\le C\biggl(\int_{X} f(x)^p\,\delta_E(x)^\beta\,d\mu(x)\biggr)^{1/p}
\end{equation}
holds for all measurable functions $f\ge 0$ in $X$.
\end{thm}

The optimality of the dimensional assumption~\eqref{e.riesz_assumption} in Theorem~\ref{t.riesz} is discussed below in Remarks~\ref{rem.optimality} and~\ref{rem.opt_HS}. The proof of Theorem \ref{t.riesz}
is based on Theorem~\ref{t.a_infty} and general two weight embedding results for Riesz potentials 
that can be found in the work of P\'erez and Wheeden~\cite{MR1962949}. More specifically,
we need the following Theorem~\ref{t.r_bounded} which is formulated here in
a slightly wider generality that we actually need;
the wider formulation is of possibly independent interest. The proof of Theorem~\ref{t.r_bounded} 
consists mainly of checking that the assumptions for the results in~\cite{MR1962949} are satisfied.  

\begin{thm}\label{t.r_bounded}
Let $s>0$. Assume that 
the reverse doubling condition \eqref{reverse doubling} holds
with the exponent $\eta= s$ and that there is $Q>s$ such that
$\mu(B)\ge c\rad(B)^Q$ for all balls $B\subset X$ with $\rad(B)\ge 1$.
Let $0<t<p\le q<\infty$, and let $w$ and $v$ be weights such that 
\[w\in \bigcup_{1\le P<\infty} A_P\,,\quad \text{ and }\quad h=v^{t/(t-p)}\in \bigcup_{1\le P<\infty} A_P\,.\]

If there exists a constant $K>0$ such that inequality 
\begin{equation}\label{e.mix}
\frac{\rad(B)^sw(B)^{t/q} h(B)^{(p-t)/p}}{\mu(B)} \le K
\end{equation}
holds for all balls $B\subset X$, then
$\mathcal{I}_s$ is bounded
from $L^{p/t}(v\,d\mu)$ to $L^{q/t}(w\,d\mu)$.
\end{thm}

\begin{proof}
Fix $f\in L^{p/t}(v\,d\mu)$.
We shall first apply \cite[Theorem 2.1]{MR1962949} that is a metric space generalization of the Euclidean result of \cite{MW}.
The former result implies that inequality
\begin{equation}\label{IandM}
\bigg( \int_{X} ( \mathcal{I}_s f)^{q/t} w \,d\mu\bigg)^{t/q} 
 \le C \bigg(\int_{X} (M_\psi f)^{q/t}\,w\,d\mu\bigg)^{t/q}
\end{equation}
holds with $C>0$ independent of $f$, where $\psi(B)=\rad(B)^s/\mu(B)$ 
for all balls $B\subset X$ and the generalized maximal function  $M_\psi f$ at $x\in X$ is defined by
$M_\psi f(x)=\sup_{B\ni x}\psi(B)\int_B|f|\,d\mu$; i.e.,
the supremum is taken over all balls $B\subset X$ that contain the point $x$.
Now, the assumptions of 
\cite[Theorem 2.1]{MR1962949} are satisfied by the following several facts.
Their somewhat tedious but straight-forward  
proofs are merely indicated below and details are left to the interested reader.

From \eqref{d.potential} we see that the Riesz potential $\mathcal I_s$
is an integral operator (transform) associated with the kernel function 
\[K(x,y)=\frac{d(x,y)^s}{\mu(B(x,d(x,y)))}\,,\qquad x,y\in X\,,\quad x\not=y\,.\]
Using both doubling \eqref{e.doubling_quant} and reverse doubling \eqref{reverse doubling} conditions of $\mu$,
the latter with $\eta= s$, one can check
that for every number $c_2> 1$ there
exists $c_1>1$ such that $K(x,y)\le c_1 K(x',y)$ if $0<d(x',y)\le c_2\,d(x,y)$ and $K(x,y)\le c_1 K(x,y')$ if $0<d(x,y')\le c_2\, d(x,y)$.
We recall from (S2) that the annuli of $X$ are non-empty. Using
this and the doubling condition of $\mu$ one shows that for every
$0<c<1$ there exists  $\Lambda>1$ such that
\begin{equation}\label{e.comparison}
\Lambda^{-1}\,\psi (B)\le  \varphi(B):=\sup\{K(x,y)\,:\,x,y\in B\text{ and }d(x,y)\ge c\rad(B)\}\le \Lambda\, \psi (B)
\end{equation}
whenever $B=B(x_B,\rad(B))\subset X$. Inequality \eqref{e.comparison} shows
that (for a fixed $0<c<1$) the maximal function $M_\psi f$ is pointwise
comparable with $M_\varphi f\colon x\mapsto \sup_{B\ni x} \varphi(B)\int_B \lvert f\rvert\,d\mu$ that appears in  \cite[Theorem 2.1]{MR1962949}.
The penultimate fact is that  $(\mu,\varphi)$ satisfies 
\cite[(16)(a)--(c)]{MR1962949} with $\tau(B)=\mu(B)$ for all balls $B\subset X$;
here one applies the doubling condition of $\mu$ and the estimates in \eqref{e.comparison}. 
The final fact required for \eqref{IandM} is that $w$ belongs to a certain ``dyadic $A_\infty$-class''
$A^{\text dy}_\infty(\mu)$.
In order to verify this condition, we refer to \cite[p.~14]{MR1962949}
and recall that the assumption $w\in \bigcup_{1\le P<\infty} A_P$ implies the
strong doubling condition \eqref{e.sdoubling} for some exponent $P\ge 1$.
Hence, we can conclude that inequality \eqref{IandM} holds.

The second step of the proof is to show
that the right-hand side of \eqref{IandM} is, in turn, 
dominated by the $L^{p/t}(v\,d\mu)$-norm of $f$,  i.e., there is $C>0$ such that
\begin{equation}\label{Mandf}
\bigg(\int_{X} (M_\psi f)^{q/t} w\,d\mu\bigg)^{t/q}
\le C\bigg(\int_{X} \lvert f\rvert^{p/t}v\,d\mu\bigg)^{t/p}.
\end{equation}
This inequality follows from \cite[Theorem 2.4]{MR1962949}, but again 
the validity of the assumptions of this theorem needs to be checked.
We will now go briefly through the necessary facts.

First, the assumed inequality \eqref{e.mix} is required in \cite[Theorem 2.4]{MR1962949}. 
The doubling condition $\psi(2B)\le 2^s\psi(B)$ for all balls $B$ is trivially valid.
The requirements \cite[(23)(a)--(b)]{MR1962949}
are the following growth conditions on balls:
\[
\psi(B_1)\le c_1\psi(B_2) \text{ if } B_1\subset B_2\subset c_2B_1 
\quad\text{ and }\quad 
\psi(B_1)\mu(B_1)\le c_1\psi(B_2)\mu(B_2) \text{ if } B_1\subset B_2.
\]
An easy application of
the doubling property of $\mu$ and inequality
$\rad(B_1)\le 2\rad(B_2)$ yields these two growth conditions.
The requirement $\lim_{\rad(B)\to \infty} \psi(B)=0$ that appears in \cite[(23)(c)]{MR1962949}
follows from the assumption that $\mu(B)\ge c\rad(B)^Q$ for all balls $B\subset X$ with $\rad(B)\ge 1$.

The last fact that we need for \eqref{Mandf} is that $h=v^{t/(t-p)}$
belongs to the 
``dyadic $A_\infty$-class'' $A^{\text dy}_\infty(\psi^{-1})$.
To this end, let us first observe that $h\,d\mu$ is a 
doubling measure, see inequality \eqref{e.sdoubling}.
The doubling and reverse doubling (with $\eta= s$) properties 
\eqref{e.doubling} and \eqref{reverse doubling} of $\mu$ imply that 
the functional $\tau=\psi^{-1}$ on balls satisfies
conditions \cite[(16)(a)--(b)]{MR1962949}.
Fix a ball $B$ and a measurable set $E\subset B$. Using our assumptions
and the inclusions \eqref{ap_relations},
we find that $h\in \bigcup_{1\le P<\infty} A_P\subset A_\infty$, and thus
\begin{equation}\label{A infty for h}
\frac{h(E)}{h(B)} 
\le C \biggl(\frac{\mu(E)}{\mu(B)}\biggr)^\delta
\le C \biggl(\frac{\rad(B)^s\Ha_{2\rad(B)}^{\mu,s}(E)}{\mu(B)}\biggr)^\delta
=C \biggl(\frac{\Ha_{2\rad(B)}^{\mu,s}(E)}{\psi^{-1}(B)}\biggr)^\delta\,.
\end{equation}
Finally, since inequality $\rad(B_j)\le 2 \rad(B)$ holds if
$B_j\subset B$ is a ball, the above
considerations show that indeed $h\in A^{\rm dy}_\infty(\psi^{-1})$;
we refer to \cite[pp.~13--15]{MR1962949} for details. 
Hence all the assumptions of \cite[Theorem 2.4]{MR1962949} are satisfied
by the above facts, and so the desired inequality \eqref{Mandf} follows from this theorem.
\end{proof}

\begin{proof}[Proof of Theorem \ref{t.riesz}]
First we note that it follows from the assumptions
that $\lcodima(E)>0$.
Indeed, if $\beta\ge 0$, this readily follows from $\lcodima(E)>\beta/(p-1)$. If
$\beta<sp-Q$, then the assumption would yield that $\lcodima(E)>Q>0$, and finally if
$sp-Q\le \beta <0$, then using the assumption $q\le Qp/(Q-sp)<\infty$ we obtain that
\[
\lcodima(E) >  Q - \frac q p (Q-sp+\beta) \ge Q - \frac Q {Q-sp} (Q-sp+\beta) > 0.
\]  
Since $\lcodima(E)>0$, it follows that $\mu(E)=0$.

For $x\in X$ we write 
\[
w(x) = \delta_{E}(x)^{(q/p)(Q-sp+\beta)-Q}\,,\qquad  v(x)=\delta_E(x)^\beta\,,\qquad h(x)=\delta_{E}(x)^{-\beta/(p-1)}\,.
\] 
Then $w$, $v$, and $h$ are all weights that belong to
the union $\bigcup_{1\le P<\infty} A_P$ of Muckenhoupt classes; 
this follows from a straight-forward calculation using  
the assumptions and Theorem~\ref{t.a_infty},
and considering the 
cases $\beta\ge 0$ and $\beta <0$ separately.

Hence it suffices to show that 
there is a constant $K>0$ such that inequality
\begin{equation}\label{e.wanted}
w(B)^{1/q} h(B)^{(p-1)/p} \le K\,\mathrm{rad}(B)^{-s}\mu(B)
\end{equation}
holds for all balls $B$ in $X$; then inequality~\eqref{e.r_weighted}
follows from Theorem~\ref{t.r_bounded} (case $t=1$).

To this end, let us
fix a ball $B=B(x_0,r)\subset X$. 
Consider first the case $\dist(B,E) <\rad(B)= r$.
Then $B\subset B(x,3r)$ for some $x\in E$.
Hence, by the fact that 
$\lcodima(E) >  Q - \frac q p (Q-sp+\beta)$,  
Remark~\ref{rmk:LT},
and the doubling condition, we obtain
\begin{align*}
w(B)^{p/q} &\le \biggl(\int_{B(x,3r)} \delta_{E}(y)^{(q/p)(Q-sp+\beta)-Q}\,d\mu(y)\biggr)^{p/q}
\le C r^{Q-sp+\beta-Qp/q}\mu(B)^{p/q}\,.
\end{align*}
Likewise, since $\lcodima(E) >  \beta/(p-1)$, we have
\begin{align*}
h(B)^{p-1} &\le  \biggl( \int_{B(x,3r)} \delta_{E}(y)^{-\beta/(p-1)} \,d\mu(y)\biggr)^{p-1} 
\le Cr^{-\beta}\mu(B)^{p-1}\,,
\end{align*}
and thus
\[
w(B)^{p/q} h(B)^{p-1} \le  
C r^{Q-sp+\beta-Qp/q}\mu(B)^{p/q} r^{-\beta}\mu(B)^{p-1} 
=C\biggl(\frac{r^Q}{\mu(B)}\biggr)^{1-p/q} \biggl(\frac{\mu(B)}{r^s}\biggr)^{p}\,.
\]
By the assumptions we have $\frac{r^Q}{\mu(B)}\le c$ and $p\le q$, 
and so inequality~\eqref{e.wanted} follows 
for all balls $B$ satisfying $\dist(B,E) < \rad(B)$.

Let us then assume that 
$\dist(B,E)\ge \rad(B)=r$. Then it holds for all $y\in B$ that
\[
 \delta_{E}(y)/3 \le  \dist(B,E) \le \delta_{E}(y)\,,
\]
and thus we may estimate 
\begin{align*}
w(B)^{p/q} h(B)^{p-1} 
&\le C \mu(B)^{p/q} \dist(B,E)^{Q-sp+\beta - Qp/q} \mu(B)^{p-1} \dist(B,E)^{-\beta}\\
& \le C \mu(B)^{p/q+p-1} \dist(B,E)^{Q-sp-Qp/q}\,.
\end{align*}
By assumption $Q-sp-Qp/q\le 0$, and so
\[
w(B)^{p/q} h(B)^{p-1} \le C\mu(B)^{p/q+p-1} r^{Q-sp-Qp/q} 
=C\biggl(\frac{r^Q}{\mu(B)}\biggr)^{1-p/q} \biggl(\frac{\mu(B)}{r^s}\biggr)^{p}.
\]
The claim now follows as in the above case $\dist(B,E) < \rad(B)$, 
and this concludes the proof.
\end{proof}

\section{Fractional Hardy--Sobolev inequalities}\label{s.fractional}

{\em{Recall our standing assumptions (S1)--(S2) concerning the
space $X$ from the beginning of \S\ref{s.weights}}.}

\medskip

We now turn to the applications of the general embeddings established in the
previous Section~\ref{s.weights}. We begin with the fractional Hardy--Sobolev 
inequalities, since these require less assumptions on the space than their classical (i.e.\ non-fractional) counterparts which, in turn, will be considered in Section~\ref{s.application}.

The main result of this section is Theorem \ref{t.hardy_Fract} that
gives a sufficient condition for the validity of fractional Hardy--Sobolev inequalities
in a metric space $X$. Even though connectivity of $X$ is not required 
due to the non-locality of these inequalities, we nevertheless need {\em some}
further structural assumptions. A suitable assumption is given by the following chain condition.

\begin{defn}\label{d.chain}
Let $\lambda \ge 1$.
We say that the space $X$ satisfies the {\em $\lambda$-chain condition}, if there
is  a constant $M\ge 1$ such that for each $x\in X$ and all $0<r<R$ 
there is a sequence of balls $B_0, B_1, B_2, \ldots, B_k$ for some integer $k$ with
the following conditions (A)--(D):
\begin{itemize}
\item[(A)] $\lambda B_0\subset X\setminus B(x,R)$ and $\lambda B_k\subset B(x,r)$,
\item[(B)] $M^{-1}\diam(\lambda B_i)\le \dist(x,\lambda B_i)\le M\diam(\lambda B_i)$
for $i=0,1,2,\ldots,k$,
\item[(C)] there is a ball $R_i\subset B_i\cap B_{i+1}$ such
that $B_i\cup B_{i+1}\subset MR_i$ for
$i=0,1,2,\ldots,k-1$,
\item[(D)] no point of $X$ belongs to more than $M$
balls $\lambda B_i$.
\end{itemize}
\end{defn}

\begin{rem}\label{r.con_chain}
If $X$ is connected, then it satisfies the $\lambda$-chain condition for all $\lambda\ge 1$, see
\cite[p.~541]{MR2569546} and \cite[p.~30]{HjK}.
Let us emphasize that $X$ need not, however, be connected
in order to satisfy a chain condition. For instance
the space $\R^n\setminus \{\lvert x\rvert=1\}$,
$n\in\N$,
equipped with the Euclidean metric and the Lebesgue measure,
is disconnected but still satisfies the $\lambda$-chain condition for all $\lambda\ge 1$ (as well as our standing assumptions (S1) and (S2)).
\end{rem}

\begin{thm}\label{t.hardy_Fract}
Let $0<s<1$. Assume that $X$ satisfies the $1$-chain condition, that 
the reverse doubling condition \eqref{reverse doubling} holds
with the exponent $\eta= s$, and that there is $Q>s$ such that
$\mu(B)\ge c\rad(B)^Q$ for all balls $B\subset X$.

Let $\emptyset \neq E\subset X$ be a closed set, and let  
$1<p\le q\le Qp/(Q-sp)<\infty$ and $\beta\in\R$ be such that
\begin{equation}\label{e.hardy_Fract_assumption}
\lcodima(E) > \max \biggl\{ Q - \frac q p (Q-sp+\beta) \, , \, \frac{\beta}{p-1}  \biggr\}\,.
\end{equation}
Then, if $1\le t <\infty$, there is a constant $C>0$ such that
the fractional Hardy--Sobolev inequality
\begin{equation}\label{e.weighted_Fract}
\begin{split}
\bigg(\int_X &\lvert f(x)\rvert^q \delta_{E}(x)^{(q/p)(Q-s p+\beta)-Q}\,d\mu(x)\bigg)^{1/q}
\\&\qquad\le C \bigg(\int_X \bigg(\int_{X}  \frac{\lvert f(y)-f(z)\rvert^t}{d(y,z)^{st} \mu(B(y,d(y,z)))}\,d\mu(z)\,\bigg)^{p/t}
\delta_{E}^\beta(y)\,d\mu(y)\bigg)^{1/p}
\end{split}
\end{equation}
holds whenever $f\in\Lip_0(X)$.
\end{thm}

We note that the case $t=q=p$ of inequality~\eqref{e.weighted_Fract} is just the (weighted) fractional Hardy inequality; see e.g.~\cite{MR3277052} and~\cite{MR3237044} for the Euclidean and metric versions of such inequalities, respectively. The case $\beta=0$ and $t=p$, i.e., non-weighted fractional Hardy--Sobolev inequality, is considered in the Euclidean case in~\cite{MR3343059}. Theorem~\ref{t.hardy_Fract}, in contrast, allows for the weighted fractional Hardy--Sobolev inequalities, and it pertains to the context of metric measure spaces.

Theorem \ref{t.hardy_Fract} follows from
Theorem \ref{t.riesz} and  Lemma \ref{l.Mak} below.
The proof of this lemma, in turn, is a modification of
\cite[Theorem 3.2]{MR2569546}.

Recall that the Riesz potential $\mathcal{I}_s$ is defined by~\eqref{d.potential}.

\begin{lem}\label{l.Mak}
Let $0<s<1$. Assume that $X$ satisfies the $1$-chain condition and that 
the reverse doubling condition \eqref{reverse doubling} holds
with the exponent $\eta= s$.

Then, if $1\le t <\infty$, we have for all $f\in \Lip_0(X)$ and all $x\in X$ that 
\begin{equation}\label{e.estim}
\lvert f(x)\rvert \le C\,\mathcal{I}_s(g)(x)\,,
\end{equation}
where $C>0$ is independent of both $f$ and $x$, and we have denoted
\begin{equation*}\label{d.g_function}
g(y)=\biggl( \int_{X}\frac{\vert f(y)-f(z)\vert^t}{d(y,z)^{st}\mu(B(y,d(y,z)))}\,d\mu(z)\biggr)^{1/t}
\end{equation*}
for every $y\in X$.
\end{lem}

\begin{proof}
Fix a function $f\in \Lip_0(X)$ and $x\in X$;
we may clearly assume that $f(x)\not=0$. 
Since $f\in \Lip_0(X)$, we can choose $R>0$ such that $f=0$ on $X\setminus B(x,R)$. 
Fix also $0<r<R$ such that
$\lvert f(x)-f(y)\rvert\le \lvert f(x)\rvert/2$
for each $y\in B(x,r)$.
The $1$-chain condition, applied with the point $x$ and $0<r<R$, now gives  
a sequence of balls $B_0, B_1, B_2,\ldots, B_k$
satisfying conditions (A)--(D) in Definition \ref{d.chain}.

By condition (A) and the fact that $f$ vanishes in $B_0\subset X\setminus B(x,R)$,
we obtain
\begin{align*}
\lvert f(x)\rvert=\vert f(x)-f_{B_0}\vert
&\le\sum_{i=0}^{k-1}\vert f_{B_{i+1}}-f_{B_{i}}\vert + \lvert f(x)-f_{B_k}\rvert\,.
\end{align*}
The above choice of $r>0$ and condition (A) together imply
that $\lvert f(x)-f_{B_k}\rvert\le \lvert f(x)\rvert/2$. Thus, by condition (C),
\begin{align*}
\lvert f(x)\rvert &\le 2\sum_{i=0}^{k-1}\vert f_{B_{i+1}}-f_{B_{i}}\vert
\le 2\sum_{i=0}^{k-1}\bigl(\vert f_{B_{i+1}}-f_{R_i}\vert
+\vert f_{B_{i}}-f_{R_i}\vert\bigr)\\
& \le C\sum_{i=0}^{k}
\vint_{B_i}\vert f(y)-f_{B_i}\vert\,d\mu(y)\,.
\end{align*}
Let us next consider a fixed $i=0,\ldots,k$. First, we have
\begin{equation*}\label{ball_est}
\begin{split}
\vint_{B_i} &\vert f(y)-f_{B_i}\vert\,d\mu(y)
\\&=\vint_{B_i}
\bigg\vert 
\frac{1}{\mu(B_i)}
\int_{B_i} ( f(y)-f(z))\,d\mu(z)\bigg\vert\,d\mu(y)\\
&\le \frac{1}{\mu(B_i)}
\int_{B_i}
\biggl(\frac{1}{\mu(B_i)} \int_{B_i}\vert f(y)-f(z)\vert^t\,d\mu(z)\biggr)^{1/t}\,d\mu(y)\\
&\le
\frac{C(2\rad(B_i))^s}{\mu(B_i)}
\int_{B_i}
\biggl( \int_{B_i}\frac{\vert f(y)-f(z)\vert^t}{(2\rad(B_i))^{st}\mu(4B_i)}\,d\mu(z)\biggr)^{1/t}\,d\mu(y)\\
&\le \frac{C\rad(B_i)^s}{\mu(B_i)}
\int_{B_i}
\biggl( \int_{B_i}\frac{\vert f(y)-f(z)\vert^t}{d(y,z)^{st}\mu(B(y,d(y,z)))}\,d\mu(z)\biggr)^{1/t}\,d\mu(y)\,.
\end{split}
\end{equation*}
Let us also fix $y\in B_i$. By condition (B),
\[
B(x,d(x,y))\subset \kappa B_i \,.
\]
where $\kappa =C(M)\ge 1$.
Hence, by the doubling~\eqref{e.doubling_quant} and
reverse doubling~\eqref{reverse doubling} conditions,
\begin{align*}
\frac{\mu(B(x,d(x,y)))}{\mu(B_i)}
\le C(\kappa,C_D)\frac{\mu(B(x,d(x,y)))}{\mu(\kappa B_i)}
\le C(\kappa,C_D,C^*)\bigg(\frac{d(x,y)}{\rad(B_i)} \bigg)^s\,.
\end{align*}
Combining the above estimates, we obtain
\begin{align*}
\sum_{i=0}^{k}
\vint_{B_i}\vert f(y)-f_{B_i}\vert \,d\mu(y)
&\le C\sum_{i=0}^{k}\int_{B_i}\frac{g(y)d(x,y)^s}{\mu(B(x,d(x,y)))}\,d\mu(y)\,.
\end{align*}
Finally, by condition (D),
\begin{equation}
\begin{split}
\vert f(x)\vert &\le C\int_{X}\frac{g(y)d(x,y)^s}{\mu(B(x,d(x,y)))}\,d\mu(y) =C\,\mathcal{I}_{s}( g)(x)\,.
\end{split}
\end{equation}
Inequality \eqref{e.estim} follows, and this concludes the proof.
\end{proof}

The next proposition shows that under some  (relatively mild)  additional assumptions,
namely Ahlfors regularity, $\lcodima(E)>0$,  $\beta \ge 0$ and $t\le q$,
the first term on the right-hand side of
\eqref{e.hardy_Fract_assumption} in Theorem~\ref{t.hardy_Fract} is optimal.

\begin{prop}\label{p.weighted_Fract_optimality}
Let $0<s<1$ and assume that $X$ is Ahlfors $Q$-regular with $Q>s$, i.e., there exists $C_A\ge 1$ such that
$C_A^{-1}r^Q \leq \mu(B(x,r))\le C_A r^Q$ for each $x\in X$ and every $0<r<\infty$. 

Let $\emptyset\not=E\subset X$ be a closed set such that $\lcodima(E)>0$,
and let $1<p\le q\le Qp/(Q-sp)<\infty$
and $\beta\ge 0$. Assume that for some $1\leq t\leq q$ and $C_{H}>0$ the fractional Hardy--Sobolev inequality
\begin{equation*}
\begin{split}
\bigg(\int_X &\lvert f(x)\rvert^q \delta_{E}(x)^{(q/p)(Q-s p+\beta)-Q}\,d\mu(x)\bigg)^{1/q}
\\&\qquad\le C_{H} \bigg(\int_X \bigg(\int_{X}  \frac{\lvert f(y)-f(z)\rvert^t}{d(y,z)^{st} \mu(B(y,d(y,z)))}\,d\mu(z)\,\bigg)^{p/t}
\delta_{E}^\beta(y)\,d\mu(y)\bigg)^{1/p}
\end{split}
\end{equation*}
holds for all functions $f\in\Lip_0(X)$.
Then
\[ 
  \lcodima(E) >  Q - \frac q p (Q-sp+\beta).
\] 
\end{prop}

\begin{proof}
We denote $\alpha = Q - \frac q p (Q-sp+\beta)$ and fix arbitrary $x\in E$ and $R>0$.
We may assume $\alpha > 0$, as otherwise there is nothing to prove.
In what follows, the varying constant $C >0$ may depend on $s$, $t$, $p$, $q$, $\beta$, $Q$ 
and the constants $C_A$ and $C_H$, but not on $x$ or $R$.

Let $f(y)=\max\{2R-d(x,y),0\}$ for each $y\in X$. Then $f\in\Lip_0(X)$
and inequality  \[\lvert f(y)-f(z)\rvert\leq \min\{d(y,z),2R\}\] holds whenever $y,z\in X$.
Thus, for any $y\in X$,
\begin{align*}
  \int_X \frac{\lvert f(y)-f(z)\rvert^t}{d(y,z)^{st} \mu(B(y, d(y,z)))}\, d\mu(z)
  &\leq
  \sum_{n\in \Z}
  \int_{2^nR \leq d(y,z) < 2^{n+1}R} \frac{\big(\min\{2^{n+1}R,2R\}\big)^t}{(2^nR)^{st} \mu(B(y, 2^n R))} \, d\mu(z) \\
  &\leq
  C \left( \sum_{n<0} R^{t-st} 2^{n(t-st)}  + \sum_{n\geq 0} R^{t-st} 2^{-nst} \right) \le C R^{t-st}.
\end{align*}
If however $y$ is far away from $x$, then this estimate can be improved. Indeed, suppose that $2^nR \leq d(x,y) < 2^{n+1}R$ for some $n\geq 2$.
Then $f(y)=0$, and if $f(z)\neq 0$ then $z \in B(x,2R)$, and so $d(y,z) \geq d(y,x) - d(x,z) \geq 2^nR - 2R \geq 2^{n-1}R$.
Consequently, 
\begin{align*}
  \int_X \frac{\lvert f(y)-f(z)\rvert^t}{d(y,z)^{st} \mu(B(y, d(y,z)))}\, d\mu(z)
  &\leq
  \int_{B(x,2R)} \frac{(2R)^t}{2^{(n-1)st}R^{st} \mu( B(y, 2^{n-1} R))} \, d\mu(z) \\
  &\leq C R^{t-st} 2^{-nst} \frac{\mu(B(x,2R))}{ \mu(B(y, 2^{n-1} R)) } \\
  &\leq C R^{t-st}2^{-n(st+Q)}.
\end{align*}

We claim that the Aikawa condition~\eqref{eq:aikawa} holds with the exponent $\alpha>0$. Indeed, since $\beta\ge 0$,
we obtain form the assumed fractional Hardy--Sobolev inequality that
\begin{align*}
  \int_{B(x,R)} \delta_E^{-\alpha}(y) \,d\mu(y)
  &\leq
  R^{-q} \int_X |f(y)|^q  \delta_E^{-\alpha}(y) \,d\mu(y)\\
  &\leq
  CR^{-q} \left[ \int_X \bigg(\int_{X}  \frac{\lvert f(y)-f(z)\rvert^t}{d(y,z)^{st} \mu(B(y,d(y,z)))}\,d\mu(z)\,\bigg)^{p/t} \delta_{E}^\beta(y)\,d\mu(y)  \right]^{q/p}\\
  &=
  CR^{-q} \left[ \left( \int_{B(x,4R)} + \sum_{n=2}^\infty \int_{2^nR \leq d(x,y) < 2^{n+1}R} \right) \bigg( \ldots \bigg)^{p/t} \delta_{E}^\beta(y)\,d\mu(y) \right]^{q/p}\\
&\leq
CR^{-q} \left[ \sum_{n=1}^\infty \mu(B(x,2^{n+1}R)) \bigg( R^{t-st}2^{-n(st+Q)} \bigg)^{p/t} (2^{n+1}R)^\beta  \right]^{q/p}\\
&\leq
C R^{-q+(q/p)(Q+p-sp+\beta)} \left[ \sum_{n=1}^\infty 2^{-n(sp+Qp/t -Q-\beta)}\right]^{q/p}
\\&\le  C R^{Q-\alpha} 
\leq C R^{-\alpha} \mu(B(x,R)),
\end{align*}
as claimed. To estimate the last series above, we used the inequality 
\[sp+Qp/t -Q-\beta \geq sp+Qp/q -Q-\beta = \alpha p/q > 0\,.\]

By Lemma~\ref{lemma:aikawa si}, there then exists some $\delta>0$ such that the Aikawa condition \eqref{eq:aikawa} holds also with the exponent $\alpha+\delta$. Thus, by Remark~\ref{rmk:LT}, $\lcodima(E) \geq \alpha + \delta > \alpha$.
\end{proof}

\begin{rem}\label{rem.optimality}
Notice that Proposition~\ref{p.weighted_Fract_optimality} also shows the sharpness of the assumption 
\[\lcodima(E) >  Q - \frac q p (Q-sp+\beta)\] in Theorem~\ref{t.riesz}. 

However, we do not know if the assumption $\lcodima(E)>\frac{\beta}{p-1}$ is optimal or even needed at all in Theorem~\ref{t.hardy_Fract}. For instance, such an extra condition is not needed in the corresponding ``thin case'' of the fractional Hardy inequalities (i.e.\ case $p=q$) considered in~\cite{MR3237044}, although there the functions are in addition assumed to vanish on $E$. Direct computations also indicate that no such condition is needed for fractional Hardy--Sobolev inequalities e.g.\ in the simple special case when $X=\R^{n-1}\times [0,\infty)$ and $E=\R^{n-1}\times\{0\}\subset\R^n$. On the other hand, we have no examples that would show the necessity of this assumption in the context of {\em fractional} Hardy--Sobolev inequalities. 

Nevertheless, in the following section we show that the corresponding assumption is indeed needed in the context of first order (i.e.\ non-fractional) Hardy--Sobolev inequalities, whence it is needed --- and in this generality also optimal --- in Theorem~\ref{t.riesz} as well; see Remark~\ref{rem.opt_HS}. Thus it seems that the bound $\lcodima(E)>\frac{\beta}{p-1}$ is in a way a built-in feature of the present approach using general $A_p$-weighted embeddings, and if one wants to get rid of this bound e.g.\ in the context of fractional Hardy--Sobolev inequalities, then a different approach needs to be used.
\end{rem}

\section{First order Hardy--Sobolev inequalities}\label{s.application}

{\em{Recall that $X$ is an unbounded metric space equipped with a doubling measure $\mu$.
The other standing assumptions from the beginning of \S\ref{s.weights}
are satisfied in this section due to the fact that the spaces considered here are necessarily  connected;  cf.\ Section \ref{s.metric}.}}

\medskip

Let us first review some basic facts concerning  upper gradients and  Poincar\'e inequalities in metric spaces. Let  $f\colon X\to \R$ be a measurable function. A Borel measurable function $g\ge 0$ on $X$
is an \emph{upper gradient} of $f$, if for all compact rectifiable curves $\gamma$ in $X$ we have
\[
  \lvert f(y)-f(x)\rvert \le \, \int_\gamma g\, ds\,.
\]
Here $x$ and $y$ are the two endpoints of $\gamma$, and the above condition should be interpreted 
as claiming that
$\int_\gamma g\, ds=\infty$ whenever at least one of $|f(x)|, |f(y)|$ is infinite. 
See e.g.~\cite{BB,HKST} for introduction on analysis on metric spaces
based on the notion of upper gradients. For instance, if $X=\R^n$
(with the Euclidean distance and the Lebesgue measure), then $g=|\nabla f|$ is
an upper gradient of a function $f\in \Lip(\R^n)$.

We say that the space $X$ supports a {\em $(1,1)$-Poincar\'e inequality} (or simply {\em Poincar\'e inequality}) 
if there exist constants 
$C_P>0$ and $\tau\ge 1$ such that whenever $B$ is a ball in $X$ and $g$ is an upper gradient of a measurable function 
$f\colon X\to \R$, we have
\begin{equation}\label{e.poincare}
    \vint_{B} |f-f_B|\, d\mu \le C_P\, \mathrm{rad}(B)\, \vint_{\tau B}g\, d\mu\,.
\end{equation}
Here the right-hand side of \eqref{e.poincare}
is interpreted as $\infty$, if the integral average
$f_B$ is not defined. 
If the space $X$ supports a Poincar\'e inequality, then $X$ is 
connected; see~\cite[Corollary 4.4]{BB}.
In particular, by Remark \ref{r.con_chain} such a space satisfies 
the
$\lambda$-chain condition given in Definition \ref{d.chain} for all $\lambda\ge 1$.

We are ready to state and prove (global)
weighted 
Hardy--Sobolev inequalities of the first order.
To the best of our knowledge, the case $q>p$ has not been considered
previously in the setting of general metric spaces. Corresponding results in $\R^n$ have been obtained in~\cite{LV}. See also the references in~\cite{LV} for some earlier results in $\R^n$ and~\cite{LHA} for results in the case $p=q$ in metric spaces. Nevertheless, Theorem~\ref{t.hardy} gives a partial improvement also to the Euclidean results of~\cite{LV}; see Remark~\ref{rem.impro}.

\begin{thm}\label{t.hardy} 
Assume that $X$  supports a $(1,1)$-Poincar\'e inequality, that 
the reverse doubling condition \eqref{reverse doubling} holds
with the exponent $\eta = 1$, and that there is $Q>1$ such that
$\mu(B)\ge c\rad(B)^Q$ for all balls $B\subset X$.

Let $\emptyset \neq E\subset X$ be a closed set, and let  
$1<p\le q\le Qp/(Q-p)<\infty$ and $\beta\in\R$ be such that 
\[
\lcodima(E) > \max \biggl\{ Q - \frac q p (Q-p+\beta) \, , \, \frac{\beta}{p-1}  \biggr\}\,.
\]
Then, there is a constant $C>0$ such that the weighted Hardy--Sobolev inequality 
\begin{equation}\label{e.weighted}
\biggl(\int_{X} \lvert f(x)\rvert^q \, \delta_E(x)^{(q/p)(Q-p+\beta)-Q}\,d\mu(x)\biggr)^{1/q}
\le C\biggl(\int_{X} g(x)^p\,\delta_E(x)^\beta\,d\mu(x)\biggr)^{1/p}
\end{equation}
holds whenever $f\in \Lip_0(X)$ and $g$ is an upper gradient of $f$.
\end{thm}

\begin{proof}
We shall adapt the line of argument from the proof of Theorem~\ref{t.hardy_Fract} from Section~\ref{s.fractional}.
Fix  $f\in \Lip_0(X)$ and its upper gradient $g$.
By Theorem \ref{t.riesz}, it suffices to prove
that there is a constant $C>0$, independent of $f$ and $g$, for which
\begin{equation}\label{e.riez_pot}
|f(x)|\le  C\mathcal{I}_{1}(g)(x)
\end{equation}
whenever $x\in X$;
recall that $\mathcal{I}_1(g)$ is defined 
by \eqref{d.potential}.
We remark that inequality \eqref{e.riez_pot} is 
essentially available in the literature, see \cite[Remark 3.3]{MR2569546}, but we
provide below some details for the sake of completeness.
 
To prove inequality \eqref{e.riez_pot}, it suffices to consider a fixed $x\in X$
for which $f(x)\not=0$. Proceeding as in the
proof of Lemma \ref{l.Mak}, but applying the
$\tau$-chain condition with $\tau\ge 1$ as in
the assumed Poincar\'e inequality \eqref{e.poincare},
we obtain $M\ge 1$ and balls $B_0,B_1,B_2,\ldots,B_k$ from Definition \ref{d.chain}
such that
\[
\lvert f(x)\rvert \le C\sum_{i=0}^k \vint_{B_i} \lvert f(y)-f_{B_i}\rvert\,d\mu(y)\,.
\]
The Poincar\'e inequality~\eqref{e.poincare} then yields 
for each $i\in \{0,1,2,\ldots,k\}$
that
\[
\vint_{B_i} \lvert f(y)-f_{B_i}\rvert\,d\mu(y) \le \frac{C\,\mathrm{rad}(B_i)}{\mu(B_i)}\int_{\tau B_i} g(y)\,d\mu(y)\,.
\]
The reverse doubling condition~\eqref{reverse doubling} with $\eta=1$ can be invoked in a similar
way as in Lemma \ref{l.Mak}, with $\kappa=C(\tau,M)\ge 1$
and with $y\in \tau B_i$ instead of $B_i$. Thereby we obtain
\begin{align*}
\sum_{i=0}^{k}
\vint_{B_i}\vert f(y)-f_{B_i}\vert \,d\mu(y)
&\le C\sum_{i=0}^{k}\int_{\tau B_i}\frac{g(y)d(x,y)}{\mu(B(x,d(x,y)))}\,d\mu(y)\,.
\end{align*}
Finally, by condition (D) in Definition \ref{d.chain},
\begin{equation}
\begin{split}
\vert f(x)\vert &\le C \int_{X}\frac{g(y)d(x,y)}{\mu(B(x,d(x,y)))}\,d\mu(y)\,,
\end{split}
\end{equation}
from which inequality \eqref{e.riez_pot} follows.
\end{proof}

Recall that in the case $q = Qp/(Q-p)$, $\beta=0$, Theorem~\ref{t.hardy} yields the (global) Sobolev inequality, which is known to hold in a metric space $X$ also under weaker assumptions than those in Theorem~\ref{t.hardy}; see, for instance~\cite[Theorem~5.50]{BB}. Nevertheless, it seems that at least some version of the measure lower bound $\mu(B)\ge c\rad(B)^Q$ is needed
for the Sobolev inequality to hold. In~\cite[Theorem~5.50]{BB}, this bound is assumed for a sequence of balls $B_j$ with $\rad(B_j)\to \infty$ as $j\to\infty$. Hence such an assumption is natural in our results as well, in particular in Theorem~\ref{t.riesz} from which all our other inequalities follow.

\begin{rem}\label{rem.impro}
Theorem~\ref{t.hardy} gives even in the Euclidean space $\R^n$ a slight improvement to the known results concerning Hardy--Sobolev inequalities. Namely, 
in~\cite[Theorem~5.1]{LV} it was proved that 
if $E$ is a closed set in $\R^n$
with $n-1\le \udima(E) < n$, and if
$1\le p\le q\le np/(n-p)<\infty$
and 
\begin{equation}\label{eq.extra}
\beta\le\frac{(p-1)(qp+np-nq)}{qp+p-q}
\end{equation}
are such that
\[
\udima(E) < \frac{q}{p}(n-p+\beta)\,, 
\]
then there is a constant $C>0$ such that inequality 
\begin{equation}\label{e.eucl.weighted}
\biggl(\int_{\R^n} \lvert f(x)\rvert^q \, \delta_E(x)^{(q/p)(n-p+\beta)-n}\,dx\biggr)^{1/q}
\le C\biggl(\int_{\R^n} \lvert \nabla f(x)\rvert^p\,\delta_E(x)^\beta\,dx\biggr)^{1/p}
\end{equation}
holds for all $f\in C^\infty_0(\R^n)$. 
Using Theorem~\ref{t.hardy}, we can now 
improve the upper bound~\eqref{eq.extra} for $\beta$. 

More precisely, the dimensional assumptions of Theorem~\ref{t.hardy} 
are in $\R^n$ equivalent to the bounds
$\udima(E)< \frac{q}{p}(n-p+\beta)$ and 
$\udima(E)< n-\beta/(p-1)$.
The latter is equivalent to
\begin{equation}\label{e.upper}
\beta<(n-\udima(E))(p-1)\,,
\end{equation} which
is a better upper bound than~\eqref{eq.extra}
if
\begin{equation}\label{eq.dim_bound}
n-\udima(E)>\frac{qp+np-nq}{qp+p-q}.
\end{equation}

But now, if $\udima(E) < \frac{q}{p}(n-p+\beta)$ and~\eqref{eq.extra} holds,
then 
\[
\udima(E) <\frac{q}{p}(n-p+\beta)\le \frac{q}{p}\biggl(n-p+\frac{(p-1)(qp+np-nq)}{qp+p-q}\biggr)
\]
and so 
\[
n-\udima(E)> 
n-\frac{nqp-qp}{qp+p-q} 
=\frac{qp+np-nq}{qp+p-q}\,.
\]
Hence~\eqref{eq.dim_bound} holds whenever the assumptions of
the case $\udima(E)\ge n-1$ of~\cite[Theorem~5.1]{LV} are satisfied, and 
we can conclude that our Theorem~\ref{t.hardy} gives in all such cases a
better upper bound for the set of admissible $\beta$ in
these Euclidean Hardy--Sobolev inequalities~\eqref{e.eucl.weighted}.

On the other hand, in the case $\udima(E) < n-1$ and $\beta>0$, \cite[Theorem~5.1]{LV} is 
better than
the present Theorem~\ref{t.hardy}, because in this case there is no additional upper bound \eqref{e.upper} for $\beta$ in~\cite{LV}.
We do not know if also in the metric setting it could be possible to get rid of, or at least weaken, 
 the assumption $\lcodima(E)>\beta/(p-1)$ when $\lcodima(E)>1$. However,  as was already mentioned in Remark~\ref{rem.optimality},
 this bound seems to be intrinsic to the present approach using $A_p$-weights, and hence other tools need to be used if one wants to weaken or remove this bound; see also the following Remark~\ref{rem.opt_HS}. In the case $\udima(E) < n-1$ of the Euclidean result~\cite[Theorem~5.1]{LV}, 
such a tool is given by Euclidean isoperimetric inequalities.
\end{rem}

\begin{rem}\label{rem.opt_HS}
Finally, let us discuss the optimality of the bound 
\[
\lcodima(E) > \max \biggl\{ Q - \frac q p (Q-p+\beta) \, , \, \frac{\beta}{p-1}  \biggr\}
\]
in Theorem~\ref{t.hardy}. In $\R^n$, the first bound is equivalent to
$\udima(E) < \frac q p (n-p+\beta)$. This is certainly optimal, since for $\beta\ge 0$ this is even necessary for the Hardy--Sobolev inequality~\eqref{e.eucl.weighted} (when $\frac q p (n-p+\beta)\neq n$); we refer to~\cite[Theorem 6.1]{LV}. The second bound reads in
$\R^n$ as $\udima(E) < n - \frac{\beta}{p-1}$, or equivalently
$\beta<(n-\udima(E))(p-1)$. Now, given any $n-1\le \lambda < n$, it is possible to
construct an Ahlfors $\lambda$-regular set $E\subset\R^n$ (so that $\udima(E)=\lambda$)
such that the Hardy--Sobolev inequality~\eqref{e.eucl.weighted} fails
whenever 
\[
\beta > p-1 \quad \bigl( \ \ge (n-\udima(E))(p-1) \bigl )\,,
\]
and the $(p,\beta)$-Hardy inequality, i.e.\ case $q=p$ in~\eqref{e.eucl.weighted}, fails also for $\beta = p-1$.
In particular, for $\udima(E)=n-1$ the bound $\beta<(n-\udima(E))(p-1)$ is sharp, showing also the sharpness of the assumption
$\lcodima(E) > \frac{\beta}{p-1}$ in Theorem~\ref{t.riesz} (as was already pointed out in Remark~\ref{rem.optimality}).

Let us give more details in the planar case; similar constuctions can be made also in higher dimensions for $n-1\le \lambda <n$, but we omit the details.

Let $E_1=\partial ([0,1]^2)\subset \R^2$ be the boundary of the unit square. Consider functions
$f_j\in C_0^\infty([0,1]^2)$ such that $f_j(x)=1$ when $\delta_{E_1}(x)\ge  2^{1-j}$,
$f_j(x)=0$ when $\delta_{E_1}(x)\le 2^{-j}$, and $|\nabla f_j|\le C 2^j$
when $2^{-j}<\delta_{E_1}(x)< 2^{1-j}$. Then, for any $1 \le  p\le q\le 2p/(2-p)<\infty$
and $\beta\in\R$, the left-hand side of the Hardy--Sobolev 
inequality~\eqref{e.eucl.weighted} is uniformly bounded away from zero 
for these functions $f_j$ if $j>1$. On the other hand, it is easy to show that
when $\beta > p-1$, the right hand side of~\eqref{e.eucl.weighted} tends to zero as $j\to\infty$, and so the Hardy--Sobolev inequality fails for all $\beta> p-1 = (2-\udima(E_1))(p-1)$
(here $n=2$ and $\udima(E_1)=1$). Moreover, for $q=p$ and $\beta = p-1$, the right-hand side of~\eqref{e.eucl.weighted} remains bounded while the left-hand side tends to infinity.
This rather simple construction already yields the sharpness of the assumption $\beta<(n-\udima(E))(p-1)$, i.e.\ 
$\lcodima(E) > \frac{\beta}{p-1}$.

To obtain a $\lambda$-regular set $E_\lambda\subset\R^2$ such that the Hardy--Sobolev inequality fails for all $\beta>p-1$ with respect to this set, we can
replace the sides of the unit square with outward pointing copies of the $\lambda$-dimensional 
``antenna set'' $K\subset\R^2$. In the complex plane, the set $K$ can be described as the unique invariant set under the 
iterated function system of similitudes $F^\alpha=\{\varphi_1,\varphi_2,\varphi_3,\varphi_4\}$, where $0<\alpha<\frac 1 2$ and 
\[\begin{split}
\varphi_1(z) = & \tfrac 1 2 z\,, \qquad  \ \ \,  \varphi_3(z)=\alpha i z + \tfrac 1 2\,,\\
\varphi_2(z) = & \tfrac 1 2 z + \tfrac 1 2\,, \quad  
         \varphi_4(z)=-\alpha i z + \tfrac 1 2 + \alpha i\,.
\end{split}\]  
Since $K=\bigcup_{j=1}^4 \varphi_j(K)$ and $F^\alpha$ satisfies the open set condition, $K$ is
$\lambda$-regular, where $1<\lambda<2$ is the solution of the equation
$2\cdot 2^{-\lambda} +  2 \alpha^\lambda = 1$; see~\cite{BT} for more details on the antenna set.

Now, from the perspective of the above test functions $f_j$ having support inside the unit square $[0,1]^2$, the set $E_\lambda$
looks just like the set $E_1$, and hence the Hardy--Sobolev inequality~\eqref{e.eucl.weighted} fails also in this case whenever 
$\beta> p-1$, and the case $q=p$ fails also when $\beta=p-1$. For $1 < \lambda < 2$ we however do not obtain sharpness of the bound $\beta<(n-\udima(E))(p-1)$ since
here $n-\udima(E)=2-\lambda < 1$. We do not know if there exist sharp examples also for the case $\udima(E)>n-1$ (or $\lcodima(E)<1$ in metric space) or if the bound $\beta<p-1$ is always optimal in this case as well.
\end{rem}

\def\cprime{$'$}

\end{document}